\documentclass{amsart}

\usepackage{amsmath, amssymb, amsthm, graphics, color}
\usepackage[backend=bibtex]{biblatex}

\usepackage{amsfonts}

\bibliography{SummerResearch}

\usepackage{ textcomp }
\newtheorem{theorem}{Theorem}

\newtheorem{lemma}[theorem]{Lemma}
\newtheorem{proposition}[theorem]{Proposition}
\newtheorem{corollary}[theorem]{Corollary}
\theoremstyle{remark}

\newcommand{\ZZ}{\mathbb{Z}}

\newcommand{\tr}[1]{\text{tr}_{L^{#1}/k}\langle 1 \rangle_{L^{#1}}}
\newcommand{\Dr}{h_{L/k}}
\newcommand{\Cr}{c^*_{L/k}}
\newcommand{\Gal}{\text{Gal}}
\newcommand{\1}{\langle1\rangle}
\newcommand{\kk}{k^\times / (k^\times)^2}

\title{Injectivity and Surjectivity of the Dress Map}
\author{Ricardo G Rojas-Echenique}

\begin{document}

\begin{abstract} 
For a nontrivial finite Galois extension $L/k$ (where the characteristic of $k$ is different
from 2) with Galois group $G$, we prove that the Dress map $\Dr : A(G) \to GW(k)$ is injective if and only if $L=k(\sqrt{\alpha})$ where $\alpha$ is not a sum of squares in $k^\times$. Furthermore, we prove that $\Dr$ is surjective if and only if $k$ is quadratically closed in $L$.  As a consequence, we give strong necessary conditions for faithfulness of the Heller-Ormsby functor $\Cr : \text{SH}_G\to \text{SH}_k$, as well as strong necessary conditions for fullness of $\Cr$. 
\end{abstract}

\maketitle

\section{Introduction}
Let $k$ be a field of characteristic different from $2$ and let $L/k$ be a finite Galois extension with Galois group $G$. Let $A(G)$ denote the Burnside ring of $G$ and let $GW(k)$ denote the Grothendieck-Witt ring of $k$. Recall that, as abelian groups, $A(G)$ is freely generated under disjoint union by cosets $G/H$ where $H$ runs through a set of  representatives for conjugacy classes of subgroups, and $GW(k)$ is generated by 1-dimensional quadratic forms $\langle a \rangle$, where $a$ runs through the group of square classes  $\kk$, under orthogonal sum $\langle a \rangle + \langle b \rangle = \langle a,b \rangle $. Multiplication in $A(G)$ is given by cartesian product with identity $G/G$ and multiplication in $GW(k)$ is given by the Kronecker product $\langle a \rangle \langle b \rangle = \langle ab \rangle $ with identity $\1$. Following the construction in \cite[Appendix B]{Dress}, the Dress map $\Dr: A(G) \to GW(k)$ is a ring homomorphism that takes the coset $G/H$ to the trace form $\tr{H}$, the quadratic form $x \mapsto \text{tr}_{L^H/k}x^2$. (Our restriction on the characteristic of $k$ is necessary for $\Dr$ to be well defined.)

A particular point of interest is that the Dress map appears naturally in the study of equivariant and motivic stable homotopy theory. Heller and Ormsby \cite[\textsection 4]{Ormsby} construct a strong symmetric monoidal triangulated functor $\Cr : \text{SH}_G\to \text{SH}_k$ from the stable $G$-equivariant homotopy category to the stable motivic homotopy category over $k$. This functor induces a homomorphism between the endomorphism rings of the unit objects in each category, which are in fact $A(G)$ and $GW(k)$, respectively. In \cite[Proposition 3.1]{Ormsby}, Heller and Ormsby show that this homomorphism agrees with $\Dr$. In particular, fullness and faithfulness of $\Cr$ are obstructed by surjectivity and injectivity of $\Dr$ respectively. 

The main goal of this note is to investigate when the Dress map, and thereby $\Cr$, is injective or surjective. While Heller and Ormsby have resolved the investigation when $\Dr$ is an isomorphism \cite[Theorem 3.4]{Ormsby}, we proceed by examining injectivity and surjectivity separately. 

When $L=k$ it is obvious that $\Dr$ is injective. The following theorem gives a complete account of when $\Dr$ is injective in the remaining cases. 

\begin{theorem}\label{main}
For a finite nontrivial Galois extension $L/k$,
$\Dr$ is injective if and only if $L=k(\sqrt{\alpha})$ where $\alpha \in k^\times$ is not a sum of squares in $k^\times$.
\end{theorem}

The proof of Theorem \ref{main} is given in \textsection \ref{inject}. Note that Theorem \ref{main}, taken with \cite[Proposition 3.1]{Ormsby}, immediately gives the following corollary. 

\begin{corollary}\label{funcinject}
If $\Cr$ is faithful, then either $L/k$ is the trivial extension or of the form described in Theorem \ref{main}.
\end{corollary}

The following theorem gives a complete account of when the Dress map is surjective. 

\begin{theorem}\label{surject}
For a finite Galois extension $L/k$, $\Dr$ is surjective if and only if $k$ is quadratically closed in $L$.
\end{theorem}

The proof of Theorem \ref{surject} is given in \textsection \ref{surjectivity}. The following corollary is immediate. 

\begin{corollary}\label{funcsurject}
If $\Cr$ is full, then $L/k$ is of the form described in Theorem \ref{surject}.
\end{corollary}

Theorems \ref{main} and \ref{surject} combine to replicate Heller and Ormsby's result that for a finite Galois extension $L/k$, $\Dr$ is an isomorphism if and only if either $k$ is quadratically closed and $L=k$, or $k$ is euclidean and $L = k(i)$. If $L/k$ is the trivial extension then Theorem \ref{surject} requires that $k$ be quadratically closed, otherwise Theorem \ref{main} requires that $L=k(\sqrt{\alpha})$ and $\kk$ contains an element that is not a sum of squares. In the latter case, $k$ must be formally real and then Theorem \ref{surject} requires that $\kk = \{(k^\times)^2, \alpha (k^\times)^2 \}$, \emph{i.e.} $|\kk| =2$, so $k$ is euclidean and $\alpha = -1$.  
 
\textbf{Acknowledgements.} I thank Kyle Ormsby for advising and editing this writeup and Irena Swanson for reviewing an earlier draft. Additionally, I thank the referee for suggesting several helpful improvements to the exposition. I gratefully acknowledge that this research was conducted with support under NSF grant DMS-1406327.

\section{Proof of Theorem \ref{main}}\label{inject}
We begin by stating a number of results that are necessary in the proof of Theorem \ref{main}. Many of these results are standard and are stated without proof. 

\begin{proposition}\label{traceforms}
Let $L/k$ be a finite Galois extension.
\begin{enumerate}
   \item If $L=k$, then $\tr{}=\1$.
   \item If $L=k(\sqrt{\alpha})$, then  
   $\tr{} = \langle 2, 2\alpha \rangle .$
   \item If $L=k(\sqrt{\alpha_1}, \sqrt{\alpha_2})$,  then 
    $\tr{} = \langle 1, \alpha_1, \alpha_2, \alpha_1\alpha_2 \rangle .$
\end{enumerate}
\end{proposition}

The following is a standard result from Galois theory.

\begin{proposition}\label{cyclic}
Let $L/k$ be a finite Galois extension with Galois group $G$. If $G \cong \ZZ/4\ZZ$, then there is a field $E$  between $L$ and $k$ such that $E=k(\sqrt{\alpha})$ where $\alpha = a^2 + b^2$ for some $a,b \in k^\times.$

\end{proposition}

The following theorem is taken directly from Lam \cite[Proposition 6.14]{Lam}. 

\begin{proposition}\label{Lam}
Let $L/k$ be a finite Galois extension, and let $E$ be any field between $k$ and $L$ with $[L:E] =2r+1$. Then 
   $$\tr{} = (2r+1)\text{tr}_{E/k}\langle 1 \rangle_E .$$ 

\end{proposition}

The following lemma is stated in different terms elsewhere. For a proof of this version see  \textsection \ref{surjectivity}. 

\begin{lemma}\label{squares}
For $\alpha \in k^{\times}$, there are positive integers $a,b$ such that $a\langle 1 \rangle = b\langle 2, 2\alpha \rangle $ if and only if  $\alpha$ is a sum of squares in $k^\times$.
\end{lemma}

We are now ready to prove Theorem \ref{main}. Suppose that $ L=k(\sqrt{\alpha})$ where $ \alpha \in k^\times$ is not a sum of squares in $k^\times$. Then the only subgroups of $\Gal(L/k)$ are the trivial subgroup and the entire group. Thus, by Proposition \ref{traceforms}, the image of $\Dr$ consists of elements of the form
\[
a\langle1\rangle + b\langle 2, 2\alpha \rangle \quad \text{ where } a,b \in \ZZ .
\]
Now suppose for contradiction that $\Dr$ is not injective. Then $\text{ker}(\Dr)$ is nontrivial. That is, 
\[
a\1 - b\langle 2, 2\alpha \rangle = 0 \quad \text{ for some } a,b \in \ZZ^+ .
\]
It follows from Lemma \ref{squares} that $\alpha$ is a sum of squares in $k^\times$. This contradicts the hypothesis so we are done with one direction. 

The other direction is more difficult so we separate the proof into lemmas. 

\begin{lemma}\label{oddprime}
Let $L/k$ be a finite Galois extension with Galois group $G$. If there is an odd prime $p$ such that $p$ divides $|G|$, then $\Dr$ is not injective. 
\end{lemma}

\begin{proof}
Suppose there is an odd prime $p$ such that $p$ divides $|G|$. Then by Cauchy's theorem there is a subgroup $H \leq G$ of order $p$, so $[L : L^H]=p$. It follows from Proposition \ref{Lam} that
\begin{equation*} 
\begin{split}
\Dr(G/e - p G/H)  & =  \tr{} -p\tr{H} \\
 & = p\tr{H} -p\tr{H} = 0
\end{split}
\end{equation*}
where $e$ is the trivial subgroup. Clearly, $e$ and $H$ are in distinct conjugacy classes thus no linear combination of $G/e$ and $G/H$ is 0 in $A(G)$. Hence, $\text{ker}(\Dr)$ is nontrivial so $\Dr$ is not injective. 
\end{proof}

\begin{lemma}\label{primepower}
Let $L/k$ be a finite Galois extension with Galois group $G$. If $|G|=2^n$ for $n > 1$ then $\Dr$ is not injective. 
\end{lemma}
\begin{proof}
Suppose $|G|=2^n$ for $n > 1$. Since a group of order $p^n$ has a normal subgroup of order $p^k$ for each $0\leq k \leq n$, $G$ has a normal subgroup $H$ of order $2^{n-2}$. We have $|\Gal(L^H/k)|=4$, so
$\Gal(L^H/k) \cong \ZZ/4\ZZ$ or $\Gal(L^H/k) \cong \ZZ/2\ZZ \times \ZZ/2\ZZ$. We now analyze each case separately. 

Suppose $\Gal(L^H/k) \cong \ZZ/4\ZZ$. Then, by Proposition \ref{cyclic}, there is a subextension $E=k(\sqrt{\alpha})$ of $L$  where $\alpha$ is a sum of two squares.  We will use the fact that if $\alpha$ is a sum of two squares, then $2 \langle \alpha \rangle = 2\langle 1 \rangle$ (a more general version of this fact is proved in \textsection \ref{surjectivity}). Now by Proposition \ref{traceforms}, and since $\alpha$ is a sum of 2 squares, 
\begin{equation*} 
\begin{split}
\Dr(4G/G - 2G/\Gal(L/E)) &=  4\1 - 2\text{tr}_{E/k}\langle 1 \rangle_E \\
  & = 4\1 - 2\langle 2, 2\alpha \rangle \\
  & = 4\1 - 4\1 = 0.
\end{split}
\end{equation*}
Clearly $G$ and $\Gal(L/E)$ are in distinct conjugacy classes. Thus $\text{ker}(\Dr)$ is nontrivial so $\Dr$ is not injective. 

Now suppose $\Gal(L^H/k) \cong \ZZ /2\ZZ \times \ZZ /2\ZZ$. Then $L^H=k(\sqrt{\alpha_1},\sqrt{\alpha_2})$ and $L^H$ has distinct subextensions $E_1 =k(\sqrt{\alpha_1})$, $E_2=k(\sqrt{\alpha_2})$, and $E_3=k(\sqrt{\alpha_1\alpha_2})$. Let $H_1=\Gal(L/E_1)$, $H_2=\Gal(L/E_2)$, and $H_3=\Gal(L/E_3)$. Consider
\begin{equation*} 
\begin{split}
  & \Dr(4G/G + 2G/H - 2G/H_1 -2G/H_2 - 2G/H_3)  \\
  &=  4\tr{G} + 2\tr{H}- 2\text{tr}_{E_1/k}\langle 1 \rangle_{E_1} - 2\text{tr}_{E_2/k}\langle 1 \rangle_{E_2} - 2\text{tr}_{E_3/k}\langle 1 \rangle_{E_3}\\
  & = 4\1 +2\langle 1, \alpha_1, \alpha_2, \alpha_1\alpha_2 \rangle - 2\langle 2, 2\alpha_1 \rangle - 2\langle 2, 2\alpha_2 \rangle- 2\langle 2, 2\alpha_1\alpha_2 \rangle\\
  & = 4\1 +2\langle 1, \alpha_1, \alpha_2, \alpha_1\alpha_2 \rangle - 2\langle 1, \alpha_1 \rangle - 2\langle 1, \alpha_2 \rangle- 2\langle 1, \alpha_1\alpha_2 \rangle\\
  & = 6\1 + 2\langle  \alpha_1, \alpha_2, \alpha_1\alpha_2 \rangle - 6\1 -  2\langle  \alpha_1, \alpha_2, \alpha_1\alpha_2 \rangle = 0.
\end{split}
\end{equation*}
Since some of the subgroups are certainly in distinct conjugacy classes the kernel of $\Dr$ is nontrivial so $\Dr$ is not injective. 
\end{proof}

With these lemmas in hand, the rest of the proof of Theorem \ref{main} is quick. Let $L/k$ be a finite nontrivial Galois extension with Galois group $G$. Suppose that $\Dr$ is injective. Then, by Lemma \ref{oddprime}, $|G|$ has no odd divisors so $|G|=2^n$ for some positive integer $n$. Furthermore, by Lemma \ref{primepower}, $|G|=2$ so $G\cong \ZZ/2\ZZ$ and $L=k(\sqrt{\alpha})$. We know that $G$ has only two subgroups, so, by Proposition \ref{traceforms}, the image of $\Dr$ is generated by linear combinations of $\1$ and $ \langle 2, 2\alpha \rangle$. Since $\Dr$ is injective we have that
\[
  a\1-b \langle 2, 2\alpha \rangle \neq 0 \quad \text{for all } a,b \in \ZZ^+.
\]
It follows from Lemma \ref{squares} that $\alpha$ is not a sum of squares, which completes the proof of Theorem \ref{main}. 
 \hfill \qedsymbol

\section{Proofs of Theorem \ref{surject} and Lemma \ref{squares}}\label{surjectivity}

\begin{proof}[Proof of Theorem \ref{surject}]
Let $L/k$ be a finite Galois extension with Galois group $G$. Suppose that $k$ is quadratically closed in $L$. It follows that $\kk$ is finite. Clearly, $L \supseteq k(\sqrt{\alpha_1},\sqrt{\alpha_2},\dots,\sqrt{\alpha_n})$ where $\{\alpha_1 (k^\times)^2, \alpha_2 (k^\times)^2, \dots, \alpha_n (k^\times)^2 \}$ generate $\kk$. For every $\alpha_i$ there is a subgroup $H_i \leq G$ such that $L^{H_i}=k(\sqrt{\alpha_i})$. 

If 2 is a square in $k$, then
\begin{equation*} 
\begin{split}
\Dr(G/H_i - G/G) &=  \tr{H_i} - \1 \\
  & = \langle 2, 2\alpha_i \rangle - \1 \\
  & = \langle \alpha_i \rangle.
\end{split}
\end{equation*}

If 2 is not a square in $k$, then $L \supseteq k(\sqrt{2})$. Let $L^U=k(\sqrt{2})$ for $U \leq G$. We have
\begin{equation*} 
\begin{split}
\Dr((G/U -G/G)G/H_i - G/G) &=  (\tr{U} -\1)\tr{H_i} - \1 \\
  & = (\langle 2,4 \rangle -\1)\langle 2, 2\alpha_i \rangle - \1 \\
  & = \langle 2 \rangle \langle 2, 2\alpha_i \rangle - \1 \\
  & = \langle \alpha_i \rangle .
\end{split}
\end{equation*}

In either case the image of $\Dr$ contains the generators of $GW(k)$ as an abelian group, so $\Dr$ must be surjective. 

Now suppose that $\Dr$ is surjective. The composition of $\text{tr}_{L^H/k}$ with the functorial map $\mathfrak{r}:GW(k) \to GW(L)$ gives $\mathfrak{r}(\text{tr}_{L^H/k}(\1_{L^H}))=n\1_L$ for any subextension $L^H$ and some positive integer $n$ (see \cite[\textsection VII 6.4]{Lam}). Thus, $\Dr$ factors through the group $GW_L^\ZZ(k)$ of $k$-quadratic forms $q$ such that $\mathfrak{r}(q) = n\1_L$ for some integer $n$. Since $\Dr$ is surjective, $GW_L^\ZZ(k) = GW(k)$, whence $\langle a \rangle_L = \1_L$ for all $a\in k^\times$. It follows that $k$ is quadratically closed in $L$.
\end{proof}

\begin{proof}[Proof of Lemma \ref{squares}]
Suppose that $a\langle 1 \rangle = b\langle 2, 2\alpha \rangle$ where $a,b \in \ZZ^+$. Then the set of elements $D(b\langle 2, 2\alpha \rangle)$ of $k^{\times}$ represented by $b\langle 2, 2\alpha \rangle$ and the set of elements $ D(a\langle 1 \rangle)$ of $k^{\times}$ represented by $a\langle 1 \rangle$ are equal.  Clearly $2\alpha \in D(b\langle 2, 2\alpha \rangle)$, so $2\alpha \in D(a\langle 1 \rangle)$. It follows that, $2\alpha$ is a sum of squares in $k^\times$. Note that the set of sums of squares in $k^\times$ is a group under multiplication so $\alpha$ is a sum of squares. 

For the other direction we proceed by induction on the stronger claim: 
\[
\alpha \text{ is a sum of $n+1$ squares } \implies 2^n \langle 1 \rangle = 2^n \langle \alpha \rangle .
\]
The base case $n=0$ is clearly true. Now assume the claim holds for any sum of $n$ squares. Then, for any sum of $n+1$ squares $\alpha = x_1^2 + \cdots + x_n^2 +x_{n+1}^2$, we have, by a standard fact about quadratic forms  (see \cite[\textsection II 4]{Lam}), 
\[
\langle x_1^2 + \cdots +x_n^2 \rangle + \langle x_{n+1}^2 \rangle = \langle \alpha \rangle ( \1 + \langle x_{n+1}^2 \rangle \langle  x_1^2 + \cdots +x_n^2 \rangle ) .
\]
If we multiply both sides of the above equality by $2^{n-1}$ , it follows from the induction hypothesis that 
\[
2^{n-1}\1 + 2^{n-1}\1 = 2^{n-1} \langle \alpha \rangle ( \1 + \1 \langle  x_1^2 + \cdots +x_n^2 \rangle )
\]
which in turn implies 
\[
2^n\1 = 2^n\langle \alpha \rangle 
\]
so we have proved the claim. 

Now note that for any even positive integer $a$, since 2 is a sum of 2 squares, 
\[
a\1 = a\langle 2 \rangle .
\]
Putting this together we see,
\[
\alpha  \text{ is a sum of $n$ squares in $k^\times$} \implies 2^n\1 = 2^{n-1} \langle2\rangle(\1 + \langle \alpha \rangle) =2^{n-1}\langle 2, 2\alpha \rangle . 
\]
This proves the other direction. 
\end{proof}

\printbibliography

\end{document}